\documentclass[12pt]{amsart}
\usepackage{amsmath,amssymb,amsbsy,amsfonts,amsthm,latexsym,
                        amsopn,amstext,amsxtra,euscript,amscd,mathrsfs,color,bm}
                        
\usepackage{float} 
\usepackage[english]{babel}
\usepackage{url}
\usepackage[colorlinks,linkcolor=blue,anchorcolor=blue,citecolor=blue,backref=page]{hyperref}
\usepackage[norefs,nocites]{refcheck}

 \usepackage[np]{numprint}

\npdecimalsign{\ensuremath{.}}

\usepackage{bibentry}

\usepackage[english]{babel}
\usepackage{mathtools}
\usepackage{todonotes}
\usepackage{url}
\usepackage[colorlinks,linkcolor=blue,anchorcolor=blue,citecolor=blue,backref=page]{hyperref}

\begin{document}

\newcommand{\commA}[2][]{\todo[#1,color=yellow]{A: #2}}
\newcommand{\commI}[2][]{\todo[#1,color=green!60]{I: #2}}
    
\newtheorem{theorem}{Theorem}
\newtheorem{lemma}[theorem]{Lemma}
\newtheorem{example}[theorem]{Example}
\newtheorem{algol}{Algorithm}
\newtheorem{corollary}[theorem]{Corollary}
\newtheorem{prop}[theorem]{Proposition}
\newtheorem{definition}[theorem]{Definition}
\newtheorem{question}[theorem]{Question}
\newtheorem{problem}[theorem]{Problem}
\newtheorem{remark}[theorem]{Remark}
\newtheorem{conjecture}[theorem]{Conjecture}

\def\xxx{\vskip5pt\hrule\vskip5pt}

\def\Cmt#1{\underline{{\sl Comments:}} {\it{#1}}}

\newcommand{\Modp}[1]{
\begin{color}{blue}
 #1\end{color}}
 
 \def\bl#1{\begin{color}{blue}#1\end{color}} 
 \def\red#1{\begin{color}{red}#1\end{color}} 

%\newcommand{\eqname}[1]{\tag{#1}}% Tag equation with name

%%%%%%%%%%%%%%%%%%%%%%%%%
% Alphabet calligraphic %
%%%%%%%%%%%%%%%%%%%%%%%%%
\def\cA{{\mathcal A}}
\def\cB{{\mathcal B}}
\def\cC{{\mathcal C}}
\def\cD{{\mathcal D}}
\def\cE{{\mathcal E}}
\def\cF{{\mathcal F}}
\def\cG{{\mathcal G}}
\def\cH{{\mathcal H}}
\def\cI{{\mathcal I}}
\def\cJ{{\mathcal J}}
\def\cK{{\mathcal K}}
\def\cL{{\mathcal L}}
\def\cM{{\mathcal M}}
\def\cN{{\mathcal N}}
\def\cO{{\mathcal O}}
\def\cP{{\mathcal P}}
\def\cQ{{\mathcal Q}}
\def\cR{{\mathcal R}}
\def\cS{{\mathcal S}}
\def\cT{{\mathcal T}}
\def\cU{{\mathcal U}}
\def\cV{{\mathcal V}}
\def\cW{{\mathcal W}}
\def\cX{{\mathcal X}}
\def\cY{{\mathcal Y}}
\def\cZ{{\mathcal Z}}

\def\C{\mathbb{C}}
\def\F{\mathbb{F}}
\def\K{\mathbb{K}}
\def\L{\mathbb{L}}
\def\G{\mathbb{G}}
\def\Z{\mathbb{Z}}
\def\R{\mathbb{R}}
\def\Q{\mathbb{Q}}
\def\N{\mathbb{N}}
\def\M{\textsf{M}}
\def\U{\mathbb{U}}
\def\P{\mathbb{P}}
\def\A{\mathbb{A}}
\def\fp{\mathfrak{p}}
\def\n{\mathfrak{n}}
\def\X{\mathcal{X}}
\def\x{\textrm{\bf x}}
\def\w{\textrm{\bf w}}
\def\a{\textrm{\bf a}}
\def\k{\textrm{\bf k}}
\def\ee{\textrm{\bf e}}
\def\ovQ{\overline{\Q}}
\def \Kab{\K^{\mathrm{ab}}}
\def \Qab{\Q^{\mathrm{ab}}}
\def \Qtr{\Q^{\mathrm{tr}}}
\def \Kc{\K^{\mathrm{c}}}
\def \Qc{\Q^{\mathrm{c}}}
\newcommand \rank{\operatorname{rk}}
\def\ZK{\Z_\K}
\def\ZKS{\Z_{\K,\cS}}
\def\ZKSf{\Z_{\K,\cS_f}}
\def\ZKSfG{\Z_{\K,\cS_{f,\Gamma}}}

\def\bF{\mathbf {F}}

\def\({\left(}
\def\){\right)}
\def\[{\left[}
\def\]{\right]}
\def\<{\langle}
\def\>{\rangle}

\def\gen#1{{\left\langle#1\right\rangle}}
\def\genp#1{{\left\langle#1\right\rangle}_p}
\def\genPs{{\left\langle P_1, \ldots, P_s\right\rangle}}
\def\genPsp{{\left\langle P_1, \ldots, P_s\right\rangle}_p}

\def\e{e}

\def\eq{\e_q}
\def\fh{{\mathfrak h}}

\def\lcm{{\mathrm{lcm}}\,}

\def\({\left(}
\def\){\right)}
\def\fl#1{\left\lfloor#1\right\rfloor}
\def\rf#1{\left\lceil#1\right\rceil}
\def\mand{\qquad\mbox{and}\qquad}

\def\jt{\tilde\jmath}
\def\ellmax{\ell_{\rm max}}
\def\llog{\log\log}

\def\m{{\rm m}}
\def\ch{\hat{h}}
\def\GL{{\rm GL}}
\def\Orb{\mathrm{Orb}}
\def\Per{\mathrm{Per}}
\def\Preper{\mathrm{Preper}}
\def \S{\mathcal{S}}
\def\vec#1{\mathbf{#1}}
\def\ov#1{{\overline{#1}}}
\def\Gal{{\mathrm Gal}}
\def\Sp{{\mathrm S}}
\def\tors{\mathrm{tors}}
\def\PGL{\mathrm{PGL}}
\def\wH{{\rm H}}
\def\Gm{\G_{\rm m}}

\def\house#1{{%
    \setbox0=\hbox{$#1$}
    \vrule height \dimexpr\ht0+1.4pt width .5pt depth \dp0\relax
    \vrule height \dimexpr\ht0+1.4pt width \dimexpr\wd0+2pt depth \dimexpr-\ht0-1pt\relax
    \llap{$#1$\kern1pt}
    \vrule height \dimexpr\ht0+1.4pt width .5pt depth \dp0\relax}}

\newcommand{\bfalpha}{{\boldsymbol{\alpha}}}
\newcommand{\bfomega}{{\boldsymbol{\omega}}}

\newcommand{\Ch}{{\operatorname{Ch}}}
\newcommand{\Elim}{{\operatorname{Elim}}}
\newcommand{\proj}{{\operatorname{proj}}}
\newcommand{\h}{{\operatorname{\mathrm{h}}}}
\newcommand{\ord}{\operatorname{ord}}

\newcommand{\hh}{\mathrm{h}}
\newcommand{\aff}{\mathrm{aff}}
\newcommand{\Spec}{{\operatorname{Spec}}}
\newcommand{\Res}{{\operatorname{Res}}}

\def\fA{{\mathfrak A}}
\def\fB{{\mathfrak B}}

\numberwithin{equation}{section}
\numberwithin{theorem}{section}

\title[On Intersections of semigroup orbits with lines]{ On intersections of polynomial semigroups orbits with plane lines}

\author[Jorge Mello]{Jorge Mello}

\address{University of New South Wales. mailing adress:\newline School of Mathematics and Statistics
UNSW Sydney 
NSW, 2052
Australia.} 

\email{j.mello@unsw.edu.au}


\begin{abstract} We study intersections of orbits in polynomial semigroup dynamics with lines on the affine plane over a number field, extending previous work of D. Ghioca, T. Tucker, M. Zieve (2008).
\end{abstract}

\maketitle
\section{Introduction}

One of the most studied topics in complex dynamics is the research on orbits of polynomial maps. For a complex number $x$ and polynomials $\mathcal{F} = \{f_1,...,f_k \} \subset \mathbb{C}[X]$, one is very interested in understanding the orbit
\begin{center}
$\mathcal{O}_{\mathcal{F}}(x)=\{ f_{i_1}(f_{i_2}(...(f_{i_n}(x))...)): n \in \mathbb{N}, i_j=1,...,k \}$
\end{center} and 
\begin{center}
 $\mathcal{F}_n= \{f_{i_1}\circ \dots \circ f_{i_n} : i_j=1,...,k \}$.
\end{center}

Considering orbits with $k=1$, D. Ghioca, T. Tucker, M. Zieve [7] proved the following

 \textit{Let $x_0, y_0 \in \mathbb{C}$ and $f,g \in \mathbb{C}[X]$ with $\deg (f)= \deg (g)>1$. If $\mathcal{O}_f(x_0) \cap \mathcal{O}_g(y_0)$ is infinite, then $f$ and $g$ have a common iterate.}

 Such result provided the first non-monomial cases of the so-called \textit{dynamical Mordell-Lang Conjecture} proposed by Ghioca and Tucker and stated below.
  \\ \\
 \textbf{ Dynamical Mordell-Lang Conjecture.} \textit{Let $f_1,...,f_k$ be polynomials in $\mathbb{C}[X]$, and let $V$ be a subvariety of the affine space $\mathbb{A}^k$ which contains no positive dimensional subvariety that is periodic under the action of $(f_1,...,f_k)$ on $\mathbb{A}^k$. Then $V(\mathbb{C})$ has finite intersection with each orbit of $(f_1,...,f_k)$ on $\mathbb{A}^k$.}\\ \\
 For an overview and a more detailed view in the history of the above conjecture, we refer to [1].
 
 The results of [7] were also extended by the same authors to function fields in the same paper, and to cases where the degrees of the polynomials are distinct in [9], in which they also generalised results to cases of a line in a higher dimensional space intersecting a product of multiple orbits defined by one map. As a corollary (Corollary 1.5, [9]) they obtained information about the intersection of a  higher dimensional line with an orbit defined by a semigroup of polynomial maps that have all but one of its coordinates as the identity.
  R. Benedetto, D. Ghioca, P. Kurlberg and T. Tucker [3] studied cases of intersection of orbits of rational functions with curves under some natural conditions, and in [4] the same authors proved that if the conjecture does not hold in the context of endomorphisms of varieties, then the set of iterates landing on a referred subvariety forms a set of density zero.
 For a discussion with an effective viewpoint and monomial maps, see [12], and for the context of finite fields, some analysis is made in [13]. On [14], intersection of orbits and the Mordell-Lang problem is studied on the disk, with non-polynomial mappings.
 
 In this paper, we study the extension of the results of [7] to polynomial semigroup cases with $k \geq 1$ over number fields under some natural conditions. Namely, for sequences $\Phi=(\phi_{i_j})_{j=1}^\infty$ of pairs of univariate polynomials  in  a finite set $\mathcal{F}$ whose \textit{coherent orbit}
 \begin{center}$\mathcal{O}^c_{\Phi}(x,y)= \{ (x,y), \phi_{i_1}(x,y), \phi_{i_1}(\phi_{i_2}(x,y)), \phi_{i_1}(\phi_{i_2}(\phi_{i_3}(x,y)),...  \}$\end{center}
 intersects the diagonal plane line $\Delta$ on infinitely many points.
 
 Among other results, we prove the following:
  \begin{theorem} Let $x_0, y_0 \in K$, and let $\mathcal{F}= \{ \phi_1,..., \phi_s \} \subset K[X] \times K[X]$ be a finite set of pairs of polynomials which are not gotten from monomials by composing with linears on both sides, with $\phi_i=(f_i,g_i)$ and $\deg f_i= \deg g_i >1$
for each $i$. Suppose that $\#(\mathcal{O}_{\mathcal{F}}(x_0,y_0) \cap \Delta) = \infty$. If the orbit $\mathcal{O}_{\mathcal{F}}(x_0,y_0)$ satisfies that there exists a sequence $\Phi$ of terms in $\mathcal{F}$ with $\# \mathcal{O}^c_{\Phi}(x,y) \cap \Delta= \infty$, or otherwise is such that $\#(\mathcal{O}_{\mathcal{F}}(x_0,y_0) \cap \Delta) = \infty$ and  the maps of $\mathcal{F}$ commute with each other, then there exist  $ k \in \mathbb{N}$, and $\phi =(f,g) \in \mathcal{F}_k$ such that $f=g$.
\end{theorem}
In Section 2 we recall properties of height functions, in Section 3 we gather very important needed results about polynomial equations and decompositions, and our main result is proved in Section 4. Further applications of the result and methods are given in Section 5.

\section{Preliminaries on height functions}

Througout the paper,  $K$ is assumed to be a fixed number field. We consider $\mathcal{F}= \{ \phi_1,..., \phi_s \} \subset K[X] \times K[X]$ to be a finite set of pairs of polynomials, with $\phi_i=(f_i,g_i)$.
Let $x,y \in K$, and let 
\begin{center}$\mathcal{O}_{\mathcal{F}}(x,y)= \{ \phi_{i_n} \circ ... \circ \phi_{i_1}(x,y) | n \in \mathbb{N} , i_j =1,...,s \}$\end{center} denote the \textit{forward orbit of $P$ under $\mathcal{F}$.}

We set $J=\{ 1,...,s \}, W= \prod_{i=1}^\infty J$, and let $\Phi_w:=(\phi_{w_j})_{j=1}^\infty$ to be a sequence of polynomials from $\mathcal{F}$ for $w= (w_j)_{j=1}^\infty \in W$. 
In this situation we let $\Phi_w^{(n)}=\phi_{w_n} \circ ... \circ \phi_{w_1}$ with $\Phi_w^{(0)}=$Id, 
and also \begin{center}$\mathcal{F}_n :=\{ \Phi_w^{(n)} | w \in W \}$.\end{center}

Precisely, we consider polynomials sequences $\Phi$ $= (\phi_{i_j})_{j=1}^\infty \in \prod_{i=1}^\infty \mathcal{F}$ and $x,y \in {K}$,
denoting \begin{center}$\Phi^{(n)}(x,y)$ $:=\phi_{i_n}(\phi_{i_{n-1}}(...(\phi_{i_1}(x,y)))$. \end{center}

The set \newline \newline $\{ (x,y), \Phi^{(1)}(x,y),  \Phi^{(2)}(x,y),  \Phi^{(3)}(x,y),... \}$ 
$ \newline \newline =\{ (x,y), \phi_{i_1}(x,y), \phi_{i_2}(\phi_{i_1}(x,y)), \phi_{i_3}(\phi_{i_2}(\phi_{i_1}(x,y)),... \}$ \newline \newline is called the \textit{forward orbit of $(x,y)$ under $\Phi$}, denoted by
$\mathcal{O}_{\Phi} (x,y)$. 

The point $(x,y)$ is said to be $\Phi$\textit{-preperiodic} if $\mathcal{O}_{\Phi} (x,y)$ is finite.

For a $x,y \in K$, the $\mathcal{F}$\textit{-orbit} of $(x,y)$ is defined as 
\begin{align*}
 \mathcal{O}_{\mathcal{F}}(x,y)=\{ \phi(x,y) | \phi \in \bigcup_{n \geq 1} \mathcal{F}_n \}&= \{ \Phi_w^{(n)}(x,y) | n \geq 0, w \in W \}\\ &= \bigcup_{w \in W} \mathcal{O}_{\Phi_w} (x,y).
 \end{align*}
 The point $(x,y)$ is called \textit{preperiodic for} $\mathcal{F}$ if $\mathcal{O}_{\mathcal{F}}(x,y)$ is finite.
 
 We let $S$ be the \textit{shift map} which sends $\Psi$ $=(\psi_i)_{i=1}^\infty$ to \begin{center}$S(\Psi) =(\psi_{i+1})_{i=1}^\infty$.\end{center}
 
 We also define the \textit{coherent orbit} of a point $(x,y)$ under a sequence $\Phi=(\phi_{i_j})_{j=1}^\infty$ to be the set
 \begin{center}
  $\mathcal{O}^c_{\Phi}(x,y)= \{ (x,y), \phi_{i_1}(x,y), \phi_{i_1}(\phi_{i_2}(x,y)), \phi_{i_1}(\phi_{i_2}(\phi_{i_3}(x,y)),...  \}$.
 \end{center}

 We  let $\Delta$ denote the diagonal line $\{ (x,x) | x \in K \}$ in the affine plane $\mathbb{A}^2(K)$.

 In order to deal with pairs of polynomials with distinct degrees, we recall known results about certain canonical heights.
 
 Recall that for $x \in \overline{\mathbb{Q}} $, the naive logarithmic height is given by 
\begin{center}$ h(x)=\sum_{v \in M_K} \dfrac{[K_v: \mathbb{Q}_v]}{[K:\mathbb{Q}]} \log(\max \{1, |x|_v\}$, \end{center}
where $M_K$ is the set of places of $K$, $M_K^\infty$ is the set of archimedean (infinite) places of $K$, $M_K^0$ is the set of nonarchimedean (finite) places of $K$, and for each $v \in M_K$, $|.|_v$ denotes the corresponding absolute value on $K$ whose restriction to $\mathbb{Q}$ gives the usual $v$-adic absolute value on $\mathbb{Q}$.
Also, we write $K_v$ for the completion of $K$ with respect to $|.|$, and we let $\mathbb{C}_v$ denote the completion of an algebraic closure of $K_v$.
 
 Considering the affine plane over a field $L$ to be $\mathbb{A}^2(L)=L \times L$, there is a way to construct height functions associated with sequences of polynomials.
 \begin{lemma} (Theorem 2.3, 10]) 
 There is a unique way to attach to each sequence $\Phi$ $=(\phi_i)_{i=1}^\infty,$ with $ \deg \phi_i \geq 2$ as above, a canonical height function
\begin{center}
 $\hat{h}_{\Phi} : \mathbb{A}^2(\bar{K}) \rightarrow \mathbb{R}$
\end{center} such that

(i) $\sup_{x \in X(\bar{K})} |\hat{h}_{\Phi}(x) - h(x)|\leq O(1)$.

(ii) $\hat{h}_{ S (\Phi)} \circ \phi_1= (\deg {\phi_1}) \hat{h}_{ \Phi}$. In particular, \begin{center}$\hat{h}_{ S^n (\Phi)} \circ \phi_n \circ ...\circ  \phi_1=(\deg{\phi_n})...(\deg{\phi_1}) \hat{h}_{ \Phi}$.\end{center}

(iii) $\hat{h}_{ \Phi}(x) \geq 0$ for all $x $.

(iv) $\hat{h}_{ \Phi}=0$ if and only if $x$ is $\Phi$-preperiodic.

We call $\hat{h}_{\Phi}$ a canonical height function (normalized) for $\Phi$.
\end{lemma}

Considering conditions as above, namely, a number field $K$, and $\mathcal{H}=\{ \phi_1,...\phi_k \}$ now with $\sum_i\deg g_i >k$, 
  the uple $(\mathbb{A}^2(K),g_1,...,g_k)$ becomes a particular case of what we call a dynamical eigensystem of degree $\deg{\phi_1}+...+ \deg{\phi_k}$.
 For such, Kawaguchi also proved the following:
 \begin{lemma}(Theorem 1.2.1, [11]) There exists the canonical height function 
 \begin{center}
  $\hat{h}_{ \mathcal{H}} : \mathbb{A}^2(\bar{K}) \rightarrow \mathbb{R}$
 \end{center} for $(X,\phi_1,...,\phi_k)$ characterized by the following two properties : 
 \newline (i) $\hat{h}_{\mathcal{H}} = h_{\mathcal{H}} + O(1) ;$ 
   \newline(ii) $\sum_{j=1}^k \hat{h}_{\mathcal{H}} \circ g_j = (\deg{g_1}+...+ \deg{g_k})\hat{h}_{\mathcal{H}}$.
  \end{lemma}
  The result below is also well known. 
 \begin{lemma} (Lemma 5.4, [7])
 If $l \in K[X]$ is linear, then there exists $c_l >0$ such that $|h(l(x))-h(x)| \leq c_l$ for all $x \in \overline{K}$.
 \end{lemma}
 \section{Some results on polynomial composition}
 The result stated below is a strong fact concerning equations of the form $F(X)=G(Y)$ with infinitely many integral solutions.
 \begin{lemma} (Corollary 2.2, [7])
  Let $K$ be a number field, $S$ a finite set of nonarchimedean places of $K$, and $F, G \in K[X]$ with $\deg(F)=\deg(G)>1$.
 Suppose $F(X)=G(Y)$ has infinitely many solutions in the ring of $S$-integers of $K$. Then $F= E \circ H \circ a$ and $G= E \circ c \circ H \circ b$ for some $E,a,b,c \in \overline{K}[X]$ with $a, b$ and $c$ linears, and $H \in \overline{K}[X]$. Moreover, for fixed $K$, there are only finitely many possibilities for $H$.
 \end{lemma}
 The next surprising result that we state shows a certain rigidity on polynomial decomposition.
 \begin{lemma}(Lemma 2.3, [7]) (Rigidity)
 Let $K$ be a field of characteristic zero. If $A, B, C, D \in K[X]-K$ satisfy $A \circ B = C \circ D$ and $\deg(B)= \deg(D)$, then there is a linear $l \in K[X]$ such that $A= C \circ l^{-1}$ and $B= l \circ D$.
  
 \end{lemma}
 Finally, we show, under some conditions, when polynomials from a finite set can be gotten from the same set through composition with linears.
 \begin{lemma}
  Let $K$ be a field of characteristic zero, and suppose $\{ F_1,...,F_h\} \subset K[X]$ is a finite set of polynomials of degree $d>1$ with the property that $u \circ F_i \circ v$ is not a monomial whenever $u,v \in K[X]$ are linear for each $i$. Then the equations $a \circ F_i = F_j \circ b$ have only finitely many solutions in linear polynomials $a,b \in K[X]$ for each $1 \leq i ,j \leq h$.
  
 \end{lemma}
 \begin{proof}
  Suppose $a \circ F_1 = F_2 \circ b$, we denote the coefficients of $X^d$ and $X^{d-1}$ in $F_1$ by $\theta_d$ and $\theta_{d-1}$, and in $F_2$ by $\tau_d$ and $\tau_{d-1}$.
  We put $\beta_1=-\theta_{d-1}/d\theta_{d}, \alpha_1= -F_1(\beta_1)$ and $\beta_2=-\tau_{d-1}/d\tau_{d}, \alpha_2= -F_2(\beta_2)$ and see that $\hat{F}_i:=\alpha_i + F_i(X + \beta_i) (i=1,2)$ have no terms of degree $d-1$ and $0$. Putting $\hat{a}:=\alpha_2 + a(x- \alpha_1)$ and $\hat{b}:= \beta_2 + b(X+ \beta_1)$, we have that $\hat{a} \circ \hat{F}_1= \hat{F}_2 \circ \hat{b}$, and both have no term of degree $d-1$.
  Hence $\hat{b}$ cannot have a term of degree $0$, neither $\hat{F}_2 \circ \hat{b}$ nor $\hat{a}$. Hence we can make $\hat{a}= \delta X$ and $\hat{b}= \gamma X$, which implies that $\delta \hat{F}_1(X)=\hat{F}_2(\gamma X)$. Writing $\hat{F}_1(X)= \sum_i u_i X^i, \hat{F}_2= \sum_i v_i X^i$, we have $\gamma^i= \delta \dfrac{u_i}{v_i}$ for each non zero $i$ term. As $\hat{F}_1, \hat{F}_2$ have at least two terms of distinct degrees, let us say $i> j$, we have $\delta^{i-j}=\dfrac{u_iv_j}{u_jv_i}$ and there are finitely many possibilities for $\delta$.
  Since by our construction $a=- \alpha_2 + \gamma^r \dfrac{v_r}{u_r}(X+ \alpha_1)$ and $b=\beta_2 + \gamma(X - \beta_1)$, there are only finitely many possilities for $a$ and $b$. Making the same procedure for any pair $(F_i, F_j)$ yields the desired result.
 \end{proof}

\section{Proof of Theorem 1.1}

\begin{proof}
 We start by letting $S$ to be a finite set of nonarchimedean places of $K$ such that the ring of $S$-integers $\mathcal{O}_S$ contains $x_0, y_0$ and every coefficient of $\phi_1,...,\phi_s$.
 Then $\mathcal{O}_S^2$ contains $\phi(x_0,y_0)$ for every $\phi \in \bigcup_{n \geq 1} \mathcal{F}_n$.

 By hypothesis we can firstly suppose $\# (\mathcal{O}_{\mathcal{F}}(x_0,y_0) \cap \Delta) = \infty$ with $\# (\mathcal{O}^c_{\Phi}(x_0,y_0) \cap \Delta) = \infty$ for some sequence $\Phi=(F,G)=(\phi_{i_j})_{j=1}^\infty= ((f_{i_j}, g_{i_j}))_{j=1}^\infty$ of terms belonging to $\mathcal{F}$,
 so that $\# \mathcal{O}_F^c(x_0)= \infty$ and $\# \mathcal{O}_G^c(y_0)= \infty$ . Let $(n_j)_{j \in \mathbb{N}}$ be such that 
 \begin{center}$\phi_{i_1}\circ...\circ \phi_{i_{n_j}}(x_0,y_0) \in \Delta$ for each $j \in \mathbb{N}$. \end{center}
 
 By the pigeonhole principle, there exists a $t_1 \in \{ 1,..., s \}$ such that for infinitely many $j$, we have that $\phi_{i_{n_j}}=\phi_{t_1}$, so that $\phi_{i_1}\circ ...\circ \phi_{i_{n_j}}(x_0,y_0)=\phi_{i_1}\circ ... \circ \phi_{t_1}(x_0,y_0) \in \Delta$. Again, for the same reason, there must exist a $\phi_{t_2} \in \mathcal{F}$ such that $\phi_{i_{n_j}}=\phi_{t_1}$,  $\phi_{i_{n_j - 1}}=\phi_{t_2}$, and $\phi_{i_1}\circ ...\circ \phi_{i_{n_j}}(x_0,y_0)=\phi_{i_1}\circ ... \circ \phi_{t_2}\circ \phi_{t_1}(x_0,y_0) \in \Delta$   for infinitely many $j$.
Obtaining $t_n$ inductively in this way, we can consider the sequence $\Phi^{\prime}=(F^{\prime},G^{\prime})=(\phi_k^{\prime})_{k=1}^\infty:=(\phi_{t_k})_{k=1}^\infty$, which by its construction satisfies that for every $k \in \mathbb{N}$, the equation ${F^{\prime}}^{(k)}(X)={G^{\prime}}^{(k)}(Y)$ has infinitely many solutions in $\mathcal{O}_S \times \mathcal{O}_S$.
By Lemma 3.1, for each $k$ we have ${F^{\prime}}^{(k)}=E_k \circ H_k \circ a_k$ and ${G^{\prime}}^{(k)}=E_k \circ c_k \circ H_k \circ b_k$ with $E_k \in \overline{K}[X]$, linears $a_k, b_k, c_k \in \overline{K}[X]$, and some $H_k \in \overline{K}[X]$ which comes from a finite set of polynomials.
Thus $H_k = H_l$ for some $k <l$.

If we write ${F^{\prime}}^{(l)}= \tilde{F} \circ {F^{\prime}}^{(k)}$ and ${G^{\prime}}^{(l)}= \tilde{G} \circ {G^{\prime}}^{(k)}$ with $(\tilde{F}, \tilde{G}) \in \mathcal{F}_{l-k}$, we have
\begin{center}
 $\tilde{F} \circ E_k \circ H_k \circ a_k= {F^{\prime}}^{(l)} = E_l \circ H_k \circ a_l $ and \newline \newline
 $\tilde{G} \circ E_k \circ c_k \circ H_k \circ b_k= {G^{\prime}}^{(l)} = E_l \circ c_l \circ H_k \circ b_l$.
\end{center} By Lemma 3.2, there are linears $u, v \in \overline{K}[X]$ such that
\begin{center}
 $H_k \circ a_k = u \circ H_k \circ a_l$ and $c_k \circ H_k \circ b_k = v \circ c_l \circ H_k \circ b_l$.
\end{center} Thus,
\begin{center}
 $\tilde{F} \circ E_k \circ u = E_l = \tilde{G} \circ E_k \circ v$,
\end{center} and by Lemma 3.2, it follows that $\tilde{F}= \tilde{G} \circ l_1$ for some $l_1 \in \overline{K}[X]$ linear.

Again from the pigeonhole principle, there exists some 
integer $k_0$ so that for infinitely many $\ell > k_0$, we have that
\begin{center}$H_{k_0} = H_\ell$\end{center}
 So, for these infinitely many $\ell$, we have 
that there exists some linear polynomial $a_\ell$ such that
\begin{center}$\phi'_{\ell}\circ \cdots \phi'_{k_0+1} := (F_{k_0,\ell}, G_{k_0, \ell})$\end{center}

satisfies

\begin{center}$F_{k_0, \ell} = G_{k_0, \ell}\circ a_\ell.$\end{center}

 Inductively, we obtain in this way an infinite $\mathcal{N} \subset \mathbb{N}$ and an infinite sequence $\Psi=(\bold{F}, \bold{G})=(\psi_i)_{i=1}^\infty$ of terms in $\bigcup_{n \geq 1} \mathcal{F}_n $ satisfying 
\begin{center} $\bold{F}^{(n)}= \bold{G}^{(n)} \circ l_n,$ with $ n \in \mathcal{N}$. \end{center} 
 By Lemma 3.2, this means  that $u_n \circ f_{i_j}= g_{i_j} \circ l_n$ for some $1 \leq i_j \leq s$ and $u_n$ linear.

Since $\{ l_n | n \in \mathcal{N} \}$ is finite by Lemma 3.3, there exist $N >n$ such that $l_N=l_n$. Then, denoting $\bold{F}^{(N)}= F_{N-n}\circ \bold{F}^{(n)}$ and $\bold{G}^{(N)}= G_{N-n}\circ \bold{G}^{(n)}$ where $F_{N-n}, G_{N-n} \in \mathcal{F}_m$ for some $m$, we have
\newline \newline
$\bold{F}^{(N)}=\bold{G}^{(N)} \circ l_N= G_{N-n}\circ \bold{G}^{(n)}\circ l_n= F_{N-n}\circ \bold{F}^{(n)}=F_{N-n}\circ \bold{G}^{(n)}\circ l_n$,
\newline \newline and thus 
\begin{center}
$F_{N-n}= G_{N-n}$
\end{center} as we wanted to show.

If otherwise, we suppose that $\# (\mathcal{O}_{\mathcal{F}}(x_0,y_0) \cap \Delta) = \infty$ and the maps of $\mathcal{F}$ commute,  we take $t_1$ such that $\phi(x_0,y_0)=(\phi_{t_1} \circ...)(x_0,y_0) \in \Delta$ for infinitely many $\phi$ if the semigroup generated by $\mathcal{F}$, so that,  $f_{t_1}(X)=g_{t_1}(Y)$ has infinitely many solutions in $\mathcal{O}_S \times \mathcal{O}_S$.
Then we choose $t_2$ such that $\phi=(\phi_{t_1} \circ \phi_{t_2} \circ ...)(x_0,y_0)=(\phi_{t_2}\circ \phi_{t_1} \circ ...)(x_0,y_0) \in \Delta$ for infinitely many $\phi \in \cup_{n \geq 0}\mathcal{F}_n$ (commutativity), so that $(f_{t_2} \circ f_{t_1})(X)=(g_{t_2} \circ g_{t_1})(Y)$ has infinitely many solutions in $\mathcal{O}_S \times \mathcal{O}_S$. In this way
we build a sequence $\Phi^{\prime}=(F^{\prime}, G^{\prime})=(\phi_{t_n})_{n \in \mathbb{N}}$ such that ${F^{\prime}}^{(k)}(X)= {G^{\prime}}^{(k)}(Y)$ has infinitely many solutions in $\mathcal{O}_S \times \mathcal{O}_S$ for each $k \in \mathbb{N}$. Then one can proceed as in the first case to achieve the desired conclusion.
\end{proof}

 It turns out that this result is actually true  for orbits which intersect an arbitrary line at infinitely many points. 
 
 \begin{corollary}
  Under the conditions of Theorem 1.1, with $L: X= l(Y)$ ( $l$ linear over $K$) in place of $\Delta$, there must exist $ k \in \mathbb{N}$, and $\phi =(f,g) \in \mathcal{F}_k$ such that $f=l\circ g$.
 \end{corollary}
\begin{proof}
 Suppose $\# (\mathcal{O}_{\mathcal{F}}(x_0,y_0) \cap L) = \infty$. Then defining a new system 
 \begin{center}
 $\mathcal{F}^l:=\{(f_1, l \circ g_1 \circ l^{-1}),..., (f_s, l \circ g_s \circ l^{-1}) \}$,
\end{center} we have that $\# (\mathcal{O}_{\mathcal{F}^l}(x_0,l(y_0)) \cap \Delta) = \infty$ with the conditions of Theorem 1.1, from where the result follows.
 \end{proof}
 \section{Further applications}
   The two next corollaries are straight-forward consequences of Theorem 1.1.
  
  \begin{corollary}
   Let $x_0, y_0 \in K$, and let $\mathcal{F}= \{ \phi_1,..., \phi_s \} \subset K[X] \times K[X]$ be a set of pairs of polynomials which  are not gotten from monomials by composing with linears on both sides, with $\phi_i=(f_i,g_i)$ and $\deg f_i= \deg g_i >1$
for each $i$. If $\# (\mathcal{O}_{\{ f_1,..., f_s \}}(x_0) \cap \mathcal{O}_{\{ g_1,..., g_s \}}(y_0) ) = \infty$ such that for some sequence $\Phi$ of terms in $\{ f_1,..., f_s \} \times \{ g_1,..., g_s \}$ we have 
\begin{center}
$\mathcal{O}_{\Phi}^c(x_0,y_0) \cap \Delta=  \infty$.
\end{center}
 Then there exists  $ k \in \mathbb{N}$, and $\phi =(f,g) \in \mathcal{F}_k$ such that $f=g$.
  \end{corollary}
  \begin{proof}
   Here we are under the conditions of Theorem 1.1, from where the result follows.
  \end{proof}

  \begin{corollary}
   Let $x_0 \in K$, and let $\mathcal{F}= \{ f_1,..., f_s \} \subset K[X]$ be a set of polynomials which  are not gotten from monomials by composing with linears on both sides, with $\deg f_1= \deg f_2=...= \deg f_s >1$. Suppose there are two sequences (trajectories) $\Phi$ and $\Psi$ in the semigroup generated by $\mathcal{F}$ satisfying the following 
   \newline \newline 
   ``$\mathcal{O}^c_{\Phi}(x_0) \cap \mathcal{O}^c_{\Psi}(x_0)= \infty$, or otherwise $\# (\mathcal{O}_{\Phi}(x_0) \cap \mathcal{O}_{\Psi}(x_0)) = \infty$ and the elements of $\mathcal{F}$ commute``.
   \newline \newline Then $\Phi$ and $\Psi$ have two ``words'' in common, namely there exist $m, k \in \mathbb{N}$ such that
   \begin{center}
    $S^m(\Phi)^{(k)}=S^m(\Psi)^{(k)}$.
   \end{center}
   \begin{proof}
    We apply the proof of Theorem 1.1 for $\mathcal{F}$ and $\mathcal{G}=\mathcal{F}$ with $(x_0,x_0)$.
   \end{proof}

  \end{corollary}

  \begin{remark}
   In a similar way as in the proof of Corollary 4.1, it can be seen that Corollary 5.1 and Corollary 5.2 can be extended with $\Delta$ being replaced by an arbitrary plane line $L: X=l(Y)$ and the set  $\{ {\Phi}^{(n_j)}(x_0)= {\Psi}^{(n_j)}(x_0) | n_1 < n_2 < ... \}$ by $\{ {\Phi}^{(n_j)}(x_0)= l({\Psi}^{(n_j)}(x_0)) | n_1 < n_2 < ... \}$ respectively, implying the more general conclusions 
   $f=l\circ g$ and $S^m(\Phi)^{(k)}=l \circ S^m(\Psi)^{(k)}$ respectively.
  \end{remark}

  Finally we obtain informations in some cases of polynomial semigroup orbits with polynomials with distinct degrees, for which the theory of canonical heights is useful.
 
\begin{prop}
 Let $x_0, y_0 \in K$, and let $\mathcal{F}= \{ \phi_1,..., \phi_s \} \subset K[X] \times K[X]$ be a set of pairs of polynomialsof degree at least $2$, and $\Phi=(F,G)$ be a sequence of terms in $\mathcal{F}$ such that $\deg(G^{(n)})= o (\deg(F^{(n)}))$. Then 
 \begin{center}

$\# (\mathcal{O}_{\Phi}(x_0,y_0) \cap \Delta) < \infty$. 
\end{center} In particular, if $\deg f_i > \deg g_i$ for each $i=1,...,s$, then every sequence( trajectory) of $\mathcal{F}$ intersects $\Delta$ in only finitely many points.
\end{prop}
\begin{proof}
 If $x_0$ or $y_0$ is preperiodic for $F$ or $G$ respectively, then the result is true. Otherwise Lemma 2.1 says that $\hat{h}_{F}(x_0)>0$, so that there is a $\delta>0$ such that every $k$ big enough satisfies
 \begin{center}
  $h(F^{(k)}(x_0))> \deg(F^{(k)})\delta$.
 \end{center} Also, there exist $\epsilon >0$ such that 
 \begin{center}
  $h(G^{(k)}(y_0))< \deg(G^{(k)})\epsilon$,
 \end{center} and by the hypothesis we know that $\deg(F^{(k)})\delta >\deg(G^{(k)})\epsilon$ for every $k$ large enough. Therefore, $h(F^{(k)}(x_0))>h(G^{(k)}(y_0))$ and $F^{(k)}(x_0) \neq G^{(k)}(y_0)$ for every $k $ large as wanted. 
\end{proof}
The result above shows in particular that if $\# (\mathcal{O}_{\Phi}(x_0,y_0) \cap \Delta) = \infty$, then it cannot be true that $\lim_n \dfrac{\deg(G^{(k)})}{\deg(F^{(k)})}$ or $\lim_n \dfrac{\deg(F^{(k)})}{\deg(G^{(k)})}$ is equal to zero, so that the distance between the canonical heights of $x_0$ and $y_0$ associated with $F$ and $G$ respectively cannot be increasingly large.

Related with Proposition 5.4 and considering the difference between the degree sum of the coordinates in the sequence, we have that if the $n$-iterates of a point under the semigroup are all contained in $\Delta$ for infinitely many $n$, then the sum of the degrees of the polynomials in the first coordinate of the generator set is equal to such sum for the second coordinate polynomials, as it follows below.
\begin{prop}
 Let $x_0, y_0 \in K$, and let $\mathcal{F}= \{ \phi_1,..., \phi_s \} \subset K[X] \times K[X]$ be a set of pairs of polynomials $\phi_i=(f_i,g_i)$ such that $\sum_i \deg f_i > \sum_i \deg g_i>s$. Suppose that $x_0$ and $y_0$ are not preperiodic for $\{ f_1,...,f_s\}$ and $\{g_1,...,g_s \}$ respectively. Then $\{ \phi(x_0,y_0) | \phi \in \mathcal{F}_n \} \not\subset \Delta$ for all but finitely many numbers $n$.
\end{prop}
\begin{proof}
 Using Lemma 2.2 and the hypothesis, we proceed similarly as in the previous proof, so that for some positive numbers $\delta$ and $\epsilon$, we have that
 \begin{align*}
   \sum_{f \in \mathcal{F}_k}h(f(x_0)) &= (\sum_i \deg f_i)^k \delta \\  &> (\sum_i \deg g_i)^k \epsilon   \\ &>\sum_{g \in \mathcal{F}_k}{h}(g(y_0))                                                                                                              
 \end{align*} for every $k$ large enough, from where the result follows.
\end{proof}

\textbf{Acknowledgements}: I am very grateful to Alina Ostafe, John Roberts and Igor Shparlinski for  
 their much helpful suggestions and comments. I also thank Professor Dragos Ghioca for valuable comments, discussions and suggestions.
I am very thankful to ARC Discovery Grant DP180100201 and UNSW for supporting me in this work.

\end{document}